\documentclass[12pt]{amsart}
\usepackage[top=70pt,bottom=70pt,left=75pt,right=77pt]{geometry}

\usepackage{amssymb, amsmath, amsthm, setspace, bbm, prettyref, graphicx, float, subfigure, color, mathrsfs, mathtools, comment}

\usepackage{Lutsko_Style}
\DeclareMathSizes{10}{10}{7}{5}

\title{Lax-Phillips Orbit Counting in Higher Rank}
\author{Alex Kontorovich and Christopher Lutsko}

\begin{document}

  \maketitle
  \begin{abstract}
    \noindent 
    Given a discrete lattice, $\Gamma < \SL_m(\R)$, and a base point $o \in \R^m$, let $N_\Gamma(T)$ denote the number of points in the orbit $o \cdot \Gamma $ whose (Euclidean) length is bounded by a growing parameter, $T$. We demonstrate an abstract spectral method \`a la Lax-Phillips, capable of  obtaining strong asymptotic estimates for $N_\Gamma(T)$, and compare and contrast it with other methods.
    %without relying on the meromorphic continuation of higher rank Langlands Eisenstein series. 
  \end{abstract}

  \onehalfspacing
  \setlength{\abovedisplayskip}{1mm}

  \section{Introduction}

    %In this paper, we study 
    Consider the general orbital counting problem in real space, by which we mean the following. Fix $m \ge 2$
    and a base point $o\in\R^m$, and
    let $\Gamma < G := \SL_m(\R)$ be a discrete lattice  such that the orbit $\cO:= o \cdot \Gamma \subset \R^m$ is discrete and infinite. Then the orbital counting problem is to obtain sharp asymptotic estimates for
    \begin{align}
        N_{\Gamma}(T) : = \# \{ z\in {\cO} \ : \ \|z\|\le T \},
    \end{align}
    where $\| z \|^2:= z_1^2+ \dots + z_m^2$ (or another archimedean norm).

There are (at least) three ``standard'' approaches to such problems. 
One is a Margulis-type argument using homogeneous dynamics \cite{Margulis1970}, 
which is very flexible and general, but often does not give the strongest error rates in applications where other methods are available, see \S\ref{ss:Comp}.
%thickening, equidistribution of expanding submanifolds, and wavefront lemmas, as in work of, e.g., Duke-Rudnick-Sarnak \cite{DukeRudnickSarnak1993} and Eskin-McMullen \cite{EskinMcMullen1993}. 
Another is Selberg/Langlands-type explicit spectral theory and meromorphic continuation of Eisenstein series \cite{Langlands1976, MoeglinWaldspurger1995}, which converts, via Tauberian arguments (see, e.g., \cite[Chapter II.7]{Tenenbaum1995}), information on location of poles into the count in question; this method often gives the strongest known error rates, but there are situations where it's difficult to obtain this explicit spectral decomposition. And a third, somewhere between the first two, is a Lax-Phillips-type approach \cite{LaxPhillips1982} using abstract spectral theory; it avoids explicit knowledge of the spectral decomposition, and often gets better error rates than homogeneous dynamics methods. 
%which works especially well in situations where the latter is difficult to obtain (such as infinite volume quotients). 
The goal of this paper is 
%not to obtain the best possible error exponent, $\eta_m$, but rather 
to exhibit
a higher-rank situation in which a Lax-Phillips approach can be adapted; as explained below, the difficulties lie 
in 
an explicit analysis of certain arising Casimir operators and Haar measures with respect to the coordinate system in our decomposition of the Lie algebra. We expect these techniques to apply in significantly more general situations, but focus on this basic case as a proof of concept and for ease of exposition.
%a soft  abstract approach to the problem, which can be further generalized to other settings where (for example) it is more difficult to access the explicit Selberg-Langlands type spectral theory. 

% \newpage

% Being a Dirichlet series with non-negative coefficients, this $E_\G(p,\cdot)$ has some abscissa of convergence $\delta\ge0$, from which one can conclude the crude estimate that 
% $$
% N_\G(T) = T^{\delta+o(1)},
% $$
% as $T\to\infty$. Since $\G$ is a lattice, thats is, has finite covolume, an area calculation yields that the exponent $\delta=m$. One can say more using spectral theory, as follows:

Whatever the method, the highest order terms come from control of the spectrum of the ring of invariant differential operators. 
The norm $\|\cdot\|$ on $\R^m$ is invariant under a maximal compact subgroup $K\cong \SO(m)$, and the locally symmetric space $X:=\Gamma\bk G/K$ is endowed with a Riemannian metric corresponding to the Killing form on the Lie algebra of $G$;
  let $\Delta$ denote the Laplace operator (quadratic Casimir) on $X$; see \S\ref{ss 22} and \S\ref{ss:irreducible} for details and normalization (and \cite{Knapp2001} for background). Then the  spectrum of $\Delta$ below $1$ consists of finitely many ``exceptional'' eigenvalues %(each with finite multiplicity)
  \begin{align*}
    0 \le  \lambda_0 < \lambda_1 \le \dots\le \lambda_k < 1.
  \end{align*}
Our normalization of the Laplace operator is such that the tempered spectrum starts at $1$. Further, let $s_i$ be the positive root $s_i = m\sqrt{\lambda_i}$. Our main result is the following

  \begin{theorem}\label{thm:m general}
    %For any lattice $\Gamma < \SL_m(\R)$, t
    There exist constants $c_0>0$, $c_1, \dots, c_k$ and $\eta_m>0$ such that
    \begin{align}
      N_\Gamma(T) =  c_0T^{m}+ c_1 T^{(m-s_1)} + \dots + c_k T^{(m-s_k)} + O( T^{(m - \eta_m)}),
     \end{align}
     where 
    \begin{equation}
        \label{eq:etamIs}         
       \eta_m= \frac{2m}{(m+2)(m-1)+4}. 
     \end{equation}
  \end{theorem}

    \subsection{Comparison to other methods}\label{ss:Comp}

    Again, we emphasize that we do not expect the method herein to give the best error term for this specific problem, but rather, we aim to exhibit a generalization of the highly flexible abstract spectral method to a higher rank counting problem. 
    %As yet, we are unaware of any sources employing the Tauberian argument or homogeneous dynamics methods to this particular problem.

By standard 
%(though lengthy) 
Tauberian arguments, 
%\cite[p 39]{IK}, 
the asymptotic expansion of $N_\G(T)$ is closely related to the  meromorphic continuation of a mirabolic-type Eisenstein (or Poincar\'e) series:
$$
E_\G : \R^m\times\C \ \ni\  (p,s)\ \mapsto \
E_\G(p,s):=\sum_{\g\in\G_o\bk\G}\frac1{\|p\g\|^s},
$$
where $\Gamma_o$ is the stabilizer of $o$ in $\Gamma$. This series converges in some half-plane $\Re(s)\gg1$ and, by Langlands' theory of Eisenstein series, has meromorphic continuation with understood relationship between the spectral theory and the location of the poles. This method is rather involved, yet it should give a linear in $m$ error term $\eta_m$ (though we have not seen it applied to this exact problem in the literature).

    % The purpose of this paper is not to obtain the best possible error exponent, $\eta_m$, but rather to exhibit a soft abstract approach to the problem, which can be further generalized to other settings where (for example) it is more difficult to access the explicit Selberg-Langlands type spectral theory. 

    The method we propose here is a generalization to higher rank of the rank-one instances of the same technique, see \cite{Kontorovich2009, KontorovichLutsko2022}, and is inspired by Lax-Phillips \cite{LaxPhillips1982}.
While this method does not recover the same exponent as the explicit spectral method of meromorphic continuation of higher-rank Eisenstein series, it avoids the technicalities thereof, while producing error exponents that are  stronger than those obtained from more traditional homogeneous dynamics approaches. The latter have proven much more flexible in situations where an explicit spectral decomposition is not readily available (see for example \cite{GorodnikNevo2012,GorodnikNevoYehoshua2017}).

Using homogeneous dynamics to attack such orbital counting problems is well studied in the literature, see, e.g., 
Margulis \cite{Margulis1970},
Duke-Rudnick-Sarnak \cite{DukeRudnickSarnak1993}, and Eskin-McMullen \cite{EskinMcMullen1993}. A typical  strategy 
is as follows. 
Writing  $H=\operatorname{Stab}_G(o)$ for the stabilizer of $o$ in $G$, let
    $$\chi_T(g)= \one_{\|o g \| < T}$$
be the indicator function of the count in question.
This is a function on $H\bk G/K$, that is, it is left-invariant by the stabilizer subgroup, and right-invariant under the maximal compact $K$, since the norm $\|\cdot\|$ is (bi-)$K$-invariant.
We then create the automorphic function:
    $$ F_T (g ) = \sum _{\g \in \G_H\bk \G }\chi_T(\gamma g),$$
where $\G_H:=\G\cap H$ is the stabilizer of $o$ in $\G$; the assumed discreteness of the orbit $\cO$ implies that $\G_H$ is a lattice in $H$. Then $F_T(e)$ is exactly the count $N_\G(T)$.
To approximate $N_\G(T)$, we smooth $F_T(e)$ as follows.
Noting that $F_T$ takes values in $\G\bk G/K$,
    we fix a bump function $\Psi$ on $\G\bk G /K$; then $F_T(e)$ is approximated by
    $$
    \wt{N}(T) = \int_{\G \bk G /K}\Psi(g)F_T(g)dg.
    $$
    Unfolding, and writing an Iwasawa-type decomposition $G=HAK$, leads to:
    $$
    \wt{N}(T) = \int_{A}\one_{\|o a\|<T} \left[\int_{\G_H\bk H} \Psi(ha) dh\right]
    da
    $$

    The bracketed term is 
the expansion by $a\in A$ of the finite $H$-volume quotient $\G_H\bk H$, and is
typically analyzed by a second smoothing process, namely, thickening and applying a wavefront-type lemma. There is significant loss in this analysis from this second smoothing.  That is, this approach involves smoothing in twice as many directions. Herein, we avoid this second smoothing.

     Our method is based on a technique using abstract spectral theory.
     In our application, the bracketed function, extended to $G$, that is, the map 
     $$
     g\mapsto 
  \int_{\G_H\bk H} \Psi(hg) dh   
$$
is a function on the double-coset $H\bk G /K$, which is one-dimensional (just a function of a one-parameter $a\in A$). Then if $\Psi$ were an eigenfunction of the quadratic Casimir operator, it would satisfy a quadratic ODE which can be solved explicitly; since the counting problem is $1$ dimensional, we need not take into consideration the full ring of invariant differential operators. This observation (see the discussion below \eqref{N Ks}), together with $L^2$-techniques developed by the authors in \cite{Kontorovich2009, KontorovichLutsko2022, Lutsko2022}, can be turned, with some effort, into a proof. 
     
     The main technical innovation in this paper is an analysis of the Lie algebra structure, explicit Casimir operator, and Haar measure  resulting from the choice of  parametrization  of the group suitable for our application. The key calculation is that, in our $H\times A\times K_H\bk K$ coordinates (see \eqref{eq:gDecomp} and Theorem \ref{thm:stru m}), the Haar measure on $G$ decomposes as $dg=dh\, da \, dk$, in which $dh$ is Haar measure on $H$. This fact is crucial for the analysis carried out in \S\ref{sec:DiffEq}.

  \begin{remark}
    In smooth form (see Theorem \ref{thm:smooth m} below), our error term exhibits square-root cancellation, that is, has size $T^{m/2}$, which is optimal in the sense that it reaches the tempered spectrum. 
  \end{remark}
     
%\begin{remark}
%    Another approach to such orbit counting problems is through the meromorphic continuation of mirobolic Eisenstein series. Indeed, applying the Mellin transform/inversion formula gives that:
%    $$
%    N_\G(T) = \sum_{\g \in \G_H \bk \G}\one_{\|o\g\|<T}
%    $$
%\end{remark}
  
    % the best known result appears to still be the ground-breaking work of . However, due to the general context in which they work, their error estimate (which is not calculated explicitly) is far from optimized. Thus, our $\eta_m$ is certainly larger.

\begin{remark}
Our proof of Theorem \ref{thm:m general} produces the error exponent given in \eqref{eq:etamIs}.
     For example $\eta_2=1/2$, $\eta_3=3/7$, $\eta_4 = 4/11$, $\eta_5 = 5/16$, and $\eta_6=3/11$. 
In the special case that $\G=\SL_m(\Z)$ and $o=e_m=(0,\dots,0,1)$, this problem amounts to counting primitive lattice points in the $T$-ball. That is, let $r_m^\ast(n)$ denote the number of ways to express an integer, $n$, as the sum of $m$ squares which share no common factor,
  \begin{align*}
    r_m^\ast(n) : = \# \{ \vect{a} \in \Z^m \ : \ \|\vect{a}\|^2 =n  \mbox{ with } (a_1, \dots, a_m)=1 \},
  \end{align*}
  where $\|\cdot \|$ denotes the Euclidean norm. 
  A classic problem is to obtain an asymptotic formula for the number
  \begin{align*}
    N_m(T):= \sum_{n=1}^{T^2} r_m^\ast(n).
  \end{align*}
  The main term is known to be $c_m T^{m}$, where $c_m= 1/\zeta(m)$ \cite{Christopher1956}, however finding optimal estimates for the error term is a challenging problem. 
  %With that, let
  %\begin{align} \label{m count prim}
  %  N_m(T) = c_m T^{m} + O(T^{m - \eta_m }).
  %\end{align}
  When $m = 2$ the best known error term is due to Wu \cite{Wu2002} (assuming the Riemann hypothesis) who shows that $\eta_2 < \frac{387}{304} \approx 1.273\dots$. For $m =3$ the best known result is that of Goldfeld-Hoffstein \cite{GoldfeldHoffstein1985} that $\eta_3 < 39/32 \approx 1.219\dots$, which follows from the fact that $N_3$ can be related to the first moment of the quadratic Dirichlet $L$-function $L(\frac{1}{2},\chi_{8d})$ (similar estimates were achieved by Young \cite{Young2009} in the smooth case). For $m \ge 4$, one can use M\"{o}bius inversion to compare the primitive lattice point count to the Gauss circle problem. Then one can easily show that the asymptotic estimate above holds for any value of $\eta_m < 1$. All of these results are much stronger that what one can achieve in the generality of Theorem \ref{thm:m general}, and are possible due to the explicit nature of the lattice $\SL_m(\Z)$. 
\end{remark}

  \subsection{Plan of paper}

    In Section \ref{s:Preliminaries}, we present some preliminaries in Lie algebras, groups,  and decompositions thereof, along with the main Structure Theorem (see Theorem \ref{thm:stru m}) for the Haar measure and Casimir operator in these coordinates. 
    %In Section \ref{s:Counting}  we prove Theorem \ref{thm:m general} in the case where $m=3$. 
    Finally, in Section \ref{s:m larger}, we 
     prove Theorem \ref{thm:m general}.
%    explain how to generalize the proof for arbitrary $m \ge 3$, the strategy being similar. 

  \section{Preliminaries}
  \label{s:Preliminaries}

Without loss of generality (conjugating $\Gamma$), we can choose our base point to be $o = e_m:= (0,\dots, 0, 1)\in \R^m$. 
 Let
  \begin{align*}
    G : = \SL_m(\R) : =
    \left\{ g  = \begin{pmatrix}
      x_{11} & x_{12} & \dots &  x_{1m}\\
      x_{21} & x_{22} & \dots &  x_{2m}\\
       & \vdots  & \vdots &       \\
      x_{(m-1)1} & x_{(m-1)2} & \dots &  x_{(m-1)m}\\
     a_1 & a_2 & \dots &  a_m
    \end{pmatrix}
    \ : \ \det{g} = 1 \right\},
  \end{align*}
  with coordinates as specified.
   Further let $H:=\operatorname{Stab}_G(o)=\{g\in G: e_m g=e_m\}$; then $H \cong \ASL_{m-1}(\R)$, or more explicitly,
  \begin{align*}
     H %: = \ASL_{m-1}(\R)
     \ = 
     \
    \left\{ g\in G \ : \ a_1 = \dots =  a_{m-1} =0 \ , \ a_m =1 \right\}.
  \end{align*}
  Now fix $\Gamma$ and let $\Gamma_H := \Gamma \cap H$. By the assumed discreteness of the orbit $\cO$, the stabilizer $\Gamma_H$ is a lattice in $H$. Then our count can be expressed as
  \begin{align*}
    N_\Gamma(T) : = \# \{ \gamma \in \Gamma_H \bk \Gamma \ : \ a_1^2 + \dots + a_m^2\le T^2 \}.
  \end{align*}

  \subsection{Group decomposition for $m=3$}

  For the reader's benefit, we  first express everything completely explicitly in the case $m=3$.
  Let $\mathfrak{g} : = \mathfrak{sl}_3(\R)$ be the Lie algebra associated to $G$. It is convenient to decompose $\mathfrak g$ according to the following basis:
  \begin{align*}
    X_{H,1} : = \begin{pmatrix} & & \\ & & 1 \\ & & \end{pmatrix}, \quad
    X_{H,2} : = \begin{pmatrix} & 1 &  \\ & &  \\ & & \end{pmatrix}, \quad
    X_{H,3} : = \begin{pmatrix} & & 1 \\ & &  \\ & & \end{pmatrix},
  \end{align*}
  \begin{align*}
    X_{H,4} : = \begin{pmatrix} 1 & &  \\ & -1 &  \\ & & \end{pmatrix}, \quad
    X_{H,5} : = \begin{pmatrix}  & 1 &  \\-1 &  &  \\ & & \end{pmatrix},
  \end{align*}
  \begin{align*}
    X_A : = \begin{pmatrix} -1/2 & & \\ & -1/2 & \\ & & 1 \end{pmatrix}
  \end{align*}
  \begin{align*}
    X_{K,1} : = \begin{pmatrix} & & 1 \\ & & \\ -1 & & \end{pmatrix}, \quad
    X_{K,2} : =\begin{pmatrix} & &  \\ & & 1 \\  & -1 & \end{pmatrix}.
  \end{align*}
  The basis elements $X_{H,i}$ generate $\mathfrak{h} =\operatorname{Lie}(H)$, and we denote their matrix exponentials by $n_1(x_1)=\exp(x_1X_{H,1} ), n_2(x_2), n_3(x_3), a_H(t), k_H(\theta)$ respectively. We denote the exponential of $X_A$ by $\wt{a}(t)=\exp(X_A t)$. Rather than work with the $t$ variable, we prefer to work with $r = e^t > 0$; thus we set
  \begin{align*}
    a(r):=\wt{a}(\log r)=\begin{pmatrix} r^{-1/2} & & \\ & r^{-1/2} & \\ & & r\end{pmatrix}.
  \end{align*}
  Finally $X_{K,1}$ and $X_{K,2}$ correspond to two rotations, and we denote their exponentials by $k_1(\theta_1)=\exp(\theta_1 X_{K,1})$ and similarly for $ k_2(\theta_2)$. Thus, given a $g \in G$, we can write
  \begin{align*}
    g = n_H(\vect{x}) a_H(t) k_H(\theta) a(r) k_1(\theta_1) k_2(\theta_2),
  \end{align*}
  where $n_H(\vect{x}) = n_1(x_1)n_2(x_2)n_3(x_3)$.
  Note that $a(r)$ commutes with $k_H(\theta)$, and if we move $k_H$ to the right, then $k_Hk_1k_2$ generate $\SO(3)$.

  % If we multiply out this expression, then we find that the bottom row of $g$ is $$r(\cos\theta_2 \cos \theta_3 , \cos \theta_2 \sin \theta_3, \sin \theta_2);$$ if we set this equal to $(a,b,c)$, then $a^2 + b^2 + c^2=r^2$. The following lemma expresses the Haar measure and (quadratic) Casimir operator in these coordinates

  % \begin{lemma}\label{lem:Haar Cas}
  %   The Haar measure on $G$ is given by
  %   \begin{align}
  %     \rd g =  r^2 e^{-2t} \cos(\theta_2) \rd \vect{x} \rd r \rd t \rd \vect{\theta}.
  %   \end{align}
  %   Moreover, the quadratic Casimir operator acting on compactly supported functions which are left $H$-invariant satisfies 
  %   \begin{align}
  %     \Delta f(a_r k_2(\theta_2) k_3(\theta_3))  &=
  %     \frac{4}{9}(r^2 \partial_{rr} + r\partial_r)f(a_r k_2(\theta_2) k_3(\theta_3)).
  %   \end{align}
  % \end{lemma}

  % \begin{proof}
  %   The proof follows the same method as in \cite[Theorem 8 \& Theorem 9]{KontorovichLutsko2022}.
  % \end{proof}

  \subsection{Group decomposition for general $m$}\label{ss 22}

  In general, we parametrize $G=\SL_m(\R)$ via the map
  $$
  H\times A \times K_H\bk K \to G
,
  $$
  where $A$ is a {\it one}-parameter diagonal group, and $K_H=K\cap H\cong \SO(m-1)$ (note that this subgroup commutes with $A$).
We furthermore decompose $H$ in standard Iwasawa coordinates,
$$
H=N_H\times A_H \times K_H,
$$
 leading to the $G$-coordinate system:
$$
  N_H\times A_H \times K_H \times A \times K_H\bk K \to G.
$$
More explicitly, we
  decompose $G$ into: $H\cong \ASL_{m-1}(\R)$  $\times$ a one-parameter diagonal subgroup $\times$ a product of $(m-1)$ one-parameter rotations. That is, we again let 
  $$a(r) = \operatorname{diag}(r^{-1/(m-1)}, \dots, r^{-1/(m-1)}, r).$$ 
  Let $X_{K,i}$ be the element of the Lie algebra with $(X_{K,i})_{mi}=-(X_{K,i})_{im}=1$ for $i = 1, \dots, m-1$, and let $k_i(\theta_i) = \exp(\theta_i X_{K,i})$. Then in a neighborhood of the identity in $G$,  we can write $g\in G$ as
  \begin{equation}\label{eq:gDecomp}
    g = h a(r) k_1(\theta_1)\cdots k_{m-1}(\theta_{m-1}),
  \end{equation}
  for some $h\in H$, $r >0$, and $\theta_i \in [0,2\pi)$. We denote by $k(\vect{\theta}) := k_1(\theta_1)\cdots k_{m-1}(\theta_{m-1})$, corresponding to a choice of Euler coordinates on the sphere $K_H\bk K \cong {\mathbb S}^{m-1}$. 

    We further decompose $H$ into a product of an upper triangular matrix $n_H(\vect{x})$, where $\vect{x}$ has dimension $\frac{m(m-1)}2$
    $\times$ a diagonal matrix $a_H(\vect{t})$ (of dimension $m-2$) $\times$ an element of $K_H\cong\SO(m-1)$ that we denote $k_H(\vect{\varphi})$ (of dimension $\frac{(m-1)(m-2)}2$). Thus we write
    \begin{align}\label{eq:coords}
      g = n_H(\vect{x}) a_H(\vect{t}) k_H(\vect{\varphi}) a(r) k(\vect{\theta}).
    \end{align}
    Crucially, note that $k_H$ commutes with $a(r)$. This allows us to multiply together the $a_H$ and $a(r)$ matrices and change coordinates to the more standard Iwasawa coordinates (see \cite{Goldfeld2006}) where the Haar measure and Casimir operator are known, resulting in the following.

    \begin{theorem}[Structure Theorem of the Haar measure and Casimir operator]
      \label{thm:stru m}
      The Haar measure on $G$ in the  coordinates of \eqref{eq:coords} is given by $dg = dh\, da\, dk$, or more explicitly:
      \begin{align}\label{Haar m}
        \rd g = 
\left(        \rho_1(\vect{x},\vect{t},\vect{\varphi})
        \rd \vect{x} \rd \vect{t} \rd \vect{\varphi} 
        \right)
        \left(
        r^{m-1} 
        \rd r 
\right)\left(
        \rho_2(\vect{\theta}) 
        \rd \vect{\theta}
        \right)
        ,
      \end{align}
      where $\rho_1$ is the Haar measure density on $\ASL_{m-1}(\R)$ and $\rho_2$ is bounded. Meanwhile the quadratic Casimir operator, acting on left-$H$-invariant and right-$K$-invariant functions $f(r)=f(ha(r)k)$, is given in these coordinates by
      \begin{align}\label{Cas m}
        \Delta f(r) = \frac{4}{m^2} (r^2 \partial_{rr} + r \partial_r) f(r).
      \end{align}
      
    \end{theorem}
    \begin{proof}
    Working in $NAK$ Iwasawa coordinates, the  Haar measure can be written as (see \cite[Theorem 1.6.1]{Goldfeld2006}) 
    \begin{align*}
        \prod_{1\le i < j \le n} d x_{i,j} \prod_{\ell=1}^{n-1} y_\ell^{-\ell(n-\ell)-1}dy_\ell dk.
    \end{align*}
    Since both decompositions end with a $K$ on right, \eqref{Haar m} follows from a change of coordinates.

      As for the Casimir operator, let $X_{H,1}, \dots, X_{H,(m-1)m}$ be a basis for $\mathfrak{h}$, let $X_A$ be the Lie element $\operatorname{diag}(\frac{1}{m-1}, \dots, \frac{1}{m-1}, 1)$, and let $X_{K,1}, \dots, X_{K,m-1}$ be the basis elements corresponding to $k_i$. 

      The quadratic Casimir operator (as an element of the universal enveloping algebra of $\frak g$) is then given by:
      \begin{align*}
        \Delta = \sum_{i=1}^{(m-1)m} X_{H,i}^\ast X_{H,i} + X_A^\ast X_A + \sum_{i=1}^{m-1}X_{K,i}^\ast X_{K,i},
      \end{align*}
      where $X^\ast$ is the dual element (that is $B(X_i,X_j^\ast)=\delta_{i,j}$ where $B$ is the Killing form). Each basis element, $X$ corresponds to a differential operator given by
      \begin{align*}
        D_Xf (g) = \frac{d}{dt} f( g\exp(t X) ) \bigg|_{t=0}.
      \end{align*}
      Using the fact that $a(r)$ commutes with almost all of $H$, one can show that the only contribution to the Casimir (when acting on $H$-invariant function) is given by $X_A^\ast X_A = c_m X_A^2$, for some constant $c_m$. Note that for the differential operator $X_A$, we can use the method in \cite[Proof of Lemma 2.7]{BourgainKontorovichSarnak2010} (also used in \cite[Proof of Theorem 8]{KontorovichLutsko2022}) to compute $X_A$. The normalization can be derived by acting on the $I$ function \cite[Definition 2.4.1]{Goldfeld2006} and matching eigenvalues. 
    \end{proof}

  \subsection{Decomposition of $L^2(\Gamma \bk G/K)$ into irreducibles and the spectral theorem}
  \label{ss:irreducible}

  The Riemannian metric on the locally symmetric space $\Gamma \bk G/K$ has an associated Laplace-Beltrami operator.
With respect to the right-regular representation of $G$ on $L^2(\Gamma \bk G)$, the quadratic Casimir operator $\Delta$ agrees on the subspace $\mathscr{H}:= L^2(\Gamma \bk G/K)$ of  right $K$-invariant functions, with the Laplacian.
  This operator, $\Delta$, is  positive and  
  self-adjoint, and thus its spectrum lies in $\R_{\ge0}$. We have the following abstract spectral theorem (see e.g., \cite[Ch. 13]{Rudin1973})

\begin{theorem}[Abstract Spectral Theorem]\label{thm:AST}
    There exists a spectral measure $\wh{\mu}$ on $\R_{\ge0}$ and a unitary spectral operator $\wh{\phantom{\cdot\cdot}}: \mathscr{H} \to L^2([0,\infty),d \wh{\mu})$ such that:
    \begin{enumerate}[label = \roman*)]
    \item Abstract Parseval's Identity: for $\phi_1,\phi_2 \in \mathscr{H}$
        \begin{align}\label{API}
            \<\phi_1,\phi_2\>_{\mathscr{H}}=\<\wh{\phi_1},\wh{\phi_2}\>_{L^2([0,\infty), d\wh{\mu})},
        \end{align}
        and
    \item The spectral operator is diagonal with respect to $L$: for $\phi\in \mathscr{H}$ and $\lambda \ge 0 $,
        \begin{equation}\label{eq:Ldiag}
            \wh{L\phi}(\lambda) = \lambda \wh{\phi}(\lambda).
        \end{equation}
    \end{enumerate}
\end{theorem}

Moreover, if $\lambda$ is in the point spectrum of $L$ with associated %orthonormal basis $\phi_\lambda^{(j)}$ of the 
eigenspace $\mathscr{H}_\gl$, then for any $\psi_1,\psi_2\in \mathscr{H}$ one has
        \begin{align}
            \wh{\psi_1}(\lambda)\wh{\bar\psi_2}(\lambda) 
            =
            \<\operatorname{Proj}_{\mathscr{H}_\gl}\psi_1,
            \operatorname{Proj}_{\mathscr{H}_\gl}\psi_2\>,
%            =
%            \sum_j \<\psi_1,\phi_\lambda^{(j)}\>\<\phi_\lambda^{(j)}, \psi_2\>.
        \end{align}
where $\operatorname{Proj}$ refers to the projection to the subspace $\mathscr{H}_\gl$. In the special case that $\mathscr{H}_\gl$ is one-dimensional and spanned by the normalized eigenfunction $\phi_\gl,$  we have that
        \begin{align}
            \wh{\psi_1}(\lambda)\wh{\bar\psi_2}(\lambda) 
            =
            \<\psi_1,\phi_\gl\>
            \<\phi_\gl,\psi_2\>.
%            =
%            \sum_j \<\psi_1,\phi_\lambda^{(j)}\>\<\phi_\lambda^{(j)}, \psi_2\>.
        \end{align}

  %% We could apply the abstract spectral theorem to each individually, however, for our purposes this is not enough. We will need to consider the joint spectrum and spectral measure coming from the unitary dual representation. 

  The group $G$ acts by right regular representation on $\mathscr{H}$. The Hilbert space $\mathscr{H}$ decomposes into components as follows
  \begin{align*}
    \mathscr{H} = \mathscr{H}_0 \oplus \mathscr{H}_1  \oplus \dots  \oplus \mathscr{H}_k  \oplus \mathscr{H}^{tempered} 
  \end{align*}
  where $\mathscr{H}_i$ is a finite dimensional eigenspace with $\Delta$-eigenvalue $\lambda_i$ and $\mathscr{H}^{tempered}$ denotes the tempered spectrum.

  \section{Proof of Theorem \ref{thm:m general}}
  \label{s:m larger}

%  In this section, we restrict to the case $m=3$. 
%Theorem \ref{thm:m general} will follow from Theorem \ref{thm:smooth} below, which achieves a similar asymptotic for a \emph{smoothed} count. 
We follow the standard procedure described in the introduction of smoothing the counting function, as follows.
Let 
  \begin{align*}
    \chi_T(g) : = \one_{\|e_m g\|<T} = \begin{cases}
      1 & \mbox{ if } r \le T \\
      0 & \mbox{ otherwise,}
    \end{cases}
  \end{align*}
where $g$ is decomposed as above into $g=n_H a_H k_H a(r)k$.
  Let
  \begin{align*}
    F_T(g) = \sum_{\gamma\in \Gamma_H \bk \Gamma}\chi_T(\gamma g),
  \end{align*}
  whence $N_\Gamma(T) = F_T(e)$.

  %% Now, it is more convenient, in the $y_2$ direction to cut-off rather than average. For this let $\Gamma_H = \Gamma_0 \cap H$. Then we let
  %% \begin{align*}
  %%   F_T(g) = \sum_{\gamma\in \Gamma_H \bk \Gamma}\chi_T(\gamma g)\chi_0(\gamma g)
  %% \end{align*}
  %% where
  %% \begin{align*}
  %%   \chi_0(g) : = \begin{cases}
  %%     1 & \mbox{ if } s \ge 1 \\
  %%     0 & \mbox{ otherwise.}
  %%   \end{cases}
  %% \end{align*}
  %% this is chosen in such a way to ensure $F_T(e) = N_\Gamma(T)$.\Chris{Check this!!!!}

For $\vep>0$, choose a smooth, nonnegative, right-$K$-invariant bump function $\psi=\psi_\vep$
supported in an $\vep$-neighborhood of the identity coset of $G/K$, with $\int_{G/K}\psi=1$, so that, for any $\g\in\G$,
\begin{equation}
    \label{eq:PsiGamma}
\int_{G/K}
\chi_T(\g g)
\psi(g)
dg
=
\begin{cases}
      1 & \mbox{ if } \|e_m\g\|<T(1-c\vep) \\
      0 & \mbox{ if } \|e_m\g\|>T(1+c\vep).
    \end{cases}
\end{equation}
Since $G/K\cong N_H\times A_H\times A$ has dimension: 
$$\frac{m(m-1)}2+(m-2)+1
=
\frac{(m+2)(m-1)}2
,$$ 
such a $\psi$ can be constructed with 
$$
\|\psi\|_{L^2(G/K)}\ll \vep^{-\frac{(m+2)(m-1)}4}.
$$
  Let $\Psi=\Psi_\vep
\in L^2(\Gamma \bk G/K)$ denote the $\Gamma$-average of $\psi_\vep$
  \begin{align*}
    \Psi(g) := \sum_{\gamma \in \Gamma} \psi(\gamma g).
  \end{align*}
It follows that
\begin{equation}
    \label{eq:PsiL2GL3}
\|\Psi\|_{L^2(\Gamma\bk G/K)}\ll \vep^{-\frac{(m+2)(m-1)}4}.
\end{equation}

  Then our smoothed count is given by
  \begin{align*}
    \wt{N}(T) : = \< F_T, \Psi\>_{L^2(\Gamma\bk G/K)}.
  \end{align*}
  For this smooth count, we have the following asymptotic.

  \begin{theorem}\label{thm:smooth m}
    For any $\Gamma < \SL_m(\Z)$ of finite co-volume we have
    \begin{align}\label{smooth m}
      \wt{N}(T) = c_0(\vep)T^{m}+ c_1(\vep) T^{m-s_1} + \dots + c_k(\vep) T^{m-s_k} + O(\vep^{-(m+2)(m-1)/4} T^{m/2}),
    \end{align}
    where for any $i = 1, \dots, k$ we have that $c_i(\vep) = C_i(1+O(\vep))$, where $C_i$ are independent of $\vep$. 
  \end{theorem}

\begin{proof}[Proof of Theorem \ref{thm:m general} from Theorem \ref{thm:smooth m}]  
After unfolding both $F_T$ and $\Psi$, we have that: 
  \begin{align*}
     \wt{N}(T) &= \sum_{\g\in\Gamma_H\bk \Gamma} \int_{G/K}  \chi_{T}(\gamma g)\psi_\vep(g) \rd g.
  \end{align*}
  It now follows from \eqref{eq:PsiGamma} that
  \begin{align*}
    \wt{N}(T(1-c\vep)) \le N_{\Gamma}(T) \le \wt{N}(T(1+c\vep)).
  \end{align*}
  From here, we optimize the parameter $\vep$ by choosing $\vep = T^{\frac{-2m}{(m+2)(m-1)+4}}$; this leads to  Theorem \ref{thm:m general} with the error term claimed in \eqref{eq:etamIs}. 
\end{proof}

The remainder of the paper is devoted to the proof of Theorem \ref{thm:smooth m}.

  \subsection{Unfolding and the differential equation}\label{sec:DiffEq}

  By unfolding, using the decomposition $g=ha(r)k$, and the calculation of Haar measure in 
 \eqref{Haar m}, 
  our smooth count becomes
  \begin{align*}
    \wt{N}(T) &= \int_{\Gamma \bk G/K} \sum_{\gamma \in \Gamma_H\bk \Gamma} \chi_T(\gamma g) \Psi(g) \rd g\\
    &= \int_{\Gamma_H \bk G/K}  \chi_T(g) \Psi(g) \rd g\\
    &= \int_0^\infty \chi_T(r)r^{m-1} \left(\int_{\Gamma_H \bk H} \Psi(ha(r)) \rd h \right)  \rd r,
  \end{align*}
  %note that 
  since $\Psi$ is right $K$-invariant. 
  Let $f(r): = \int_{\Gamma_H \bk H} \Psi(ha(r)) \rd h$ denote the quantity inside the brackets.

  Then using Theorem \ref{thm:stru m} we know that, for any value of $\lambda$, $f$ satisfies the differential equation
  \begin{equation}\label{eq:fDiffEq}
    \left(\frac{4}{m^2}(r^2\partial_{rr} + r\partial_r) - \lambda\right)f(r)= g(r)
  \end{equation}
  with $$g(r) :  = \int_{\Gamma_H \bk H} (\Delta -\lambda) \Psi( h a(r )) \rd h.$$
  %The same equation holds for any $\Psi \in L^2(\Gamma \bk G / K)$ and any value $\lambda$. 
  For $\lambda\neq0$,  the homogeneous case  ($g \equiv 0$) in \eqref{eq:fDiffEq}
  has two  solutions, namely $f_\pm(r)=A_\pm\, r^{m-1\pm s}$, for some constants $A_\pm$; here we have written $\lambda = \frac{4}{m^2}s^2$ with $s>0$. 
 Write
  \begin{align*}
    \alpha_{\pm}(T) &: =  \int_0^\infty \chi_T(r) r^{m-1\pm s}  \rd r \\
    & = \frac{1}{m\pm s} T^{m\pm s}
    .
  \end{align*}
  %since $\abs{\Re(s)} \le 3/2$. 
  Using similar lines of proof of \cite[Proof of Lemma 3.3]{Kontorovich2009}, we can thus deduce
  \begin{lemma} \label{lem:K3.3}
    For any right $K$ invariant function $\Psi$, $T>0$ large enough and any $\lambda >0$, we have
      \begin{align}
        \<F_T,\Psi\> =  A_+ \alpha_+(T)+A_-\alpha_-(T) + O(\|(\Delta- \lambda)\Psi\|).
    \end{align}
    \end{lemma}
    \begin{proof}
        Using the solutions to the homogeneous case, $f_{\pm}$ we can derive the solutions to the inhomogeneous cases using the method of variation of parameters. With that, we arrive at the inhomogeneous solution
        \begin{align*}
            f(r) = f_+(r)+f_-(r) + u(r)r^{m-1+s} +v(r)r^{m-1-s},
        \end{align*}
        where 
        \begin{align*}
            u(r): =  \int_r^{T} \frac{g(t)}{t^{m-2+s}}\rd t
            \qquad
            \text{ and }
            \qquad
            v(r): =  \int_r^{T} \frac{g(t)}{t^{m-2-s}}\rd t
        \end{align*}
        Integrating $f$ against $\chi_T(r)$ from $0$ to $\infty$ yields
        \begin{align*}
            \<F_T,\Psi\> = A_+\alpha_+(T) + A_1\alpha_-(T) + I + II,
        \end{align*}
        where
        \begin{align*}
            I: = \int_0^\infty \chi_T(r) r^{m-1+s}u(r) \rd r \qquad
            \text{ and }
            \qquad
            II: = \int_0^\infty \chi_T(r) r^{m-1-s}v(r) \rd r.
        \end{align*}
        Next, we apply integration by parts, yielding
        \begin{align*}
             I = \left[\frac{r^{ m+s}u(r)}{m+s}\right]^{T}_0 + \frac{1}{m+s} \int_0^T r^{2s+2}g(r)\rd r.
        \end{align*}
        The boundary term is then $0$ since $u(T)=0$ (the same identity holds for $II$)

        Finally, opening up the definition of $g$, using that $\Gamma \bk G$ is finite, and applying the triangle inequality and integration by parts (as in \cite[p 33]{Kontorovich2009}) we conclude
        \begin{align*}
            \max(I,II) = O(\|(\Delta-\lambda)\Phi\|).
        \end{align*}
        
    \end{proof}

  However, note that we can trivially bound $\wt{N}(T) \ll T^{m}$, hence $A_+=0$. Thus we can in fact write
  \begin{align}
    F_T =  A \alpha(T) + O(\|(\Delta- \lambda)\Psi\|),
  \end{align}
  where $\alpha(T) = \alpha_-(T)$ and $A=A_-$. Hence we can solve for $A$ and write
  \begin{align}\label{N Ks}
    \<F_T,\Psi\> = K_T(\lambda)\<F_1,\Psi\> + O(\|(\Delta - \lambda)\Psi\|),
  \end{align}
  with $K_T(\lambda) = \frac{\alpha(T)}{\alpha(1)}$. The following theorem states that, since \eqref{N Ks} holds for any $\Psi$ and any $\lambda$, we can in fact show that the error vanishes. That is, \eqref{N Ks} shows that, if $\Psi$ were an eigenfunction of the quadratic Casimir operator, then $K_T$ grows the ball of radius $1$ to the ball of radius $T$. Since this holds for all $\lambda$, we can choose a particular test function and show that the error term is equal to $0$. This is the key step present in Lax-Phillips's argument.        
  
    Note when considering a function of an operator, such as $K_{T}(\Delta)$, we define it via a power series expansion of $K_{T}$.
  \begin{theorem}[Main Identity]\label{thm:Main Ident}
    For $T$ large enough we have
    \begin{align}
      F_T(g) =  K_{T}(\Delta)F_{1}(g)
    \end{align}
    almost everywhere. Moreover
    \begin{align}\label{K bounds}
      K_{T}(\lambda) = \begin{cases}
        cT^{m-s} & \mbox{ if } s < m/2 \\
        cT^{m/2} & \mbox{ if } s =m/2+it.
        \end{cases}
    \end{align}
  \end{theorem}

  \begin{proof}
      Since many of the detail can be found in \cite[Prop 3.5 + Theorem 3.2]{Kontorovich2009}, we will prove this statement for $s<m/2$.
      As in \eqref{N Ks}, we can write
      \begin{align*}
         \<F_T,\phi\> = K_T(\lambda)\<F_1,\phi\> + O(\|(\Delta - \lambda)\phi\|).
      \end{align*}
      Consider the function
      \begin{align*}
          G_T &: = F_T - K_T(\Delta)F_1.
      \end{align*}
    By adding and subtracting like terms, for any $K$-invariant function $\phi$, we write
    \begin{align}
    \begin{aligned}
        \label{GT exp}
        \<G_T, \phi\> = \left(\<F_T, \phi\> 
                      - cT^{m-s}\right)
                      &-\left(\<K_T(\Delta)F_1,\phi\> - K_T(\lambda)\<F_1, \phi\>\right)\\
                      & -\left(K_T(\lambda)\<F_1,\phi\>-K_T(\lambda)c\right)\\
                      &  -\left(K_T(\lambda)c -cT^{m-s} \right).
    \end{aligned}
    \end{align}
    We claim that each of the terms in the brackets on the right hand side is $O( \|(\Delta-\lambda)\phi \|)$. The bottom row follows from the definition of $\alpha$ and $K_T$. The second from the bottom row is straightforward. The second term in the top row requires the use of the spectral theorem and the mean value theorem (see \cite[p 16]{Kontorovich2009} for a general argument) and the first bracket is bounded by Lemma \ref{lem:K3.3}.
    Thus we have the inequality
    \begin{align}\label{GT bound}
        \<G_T,\phi\> = O( \|(\Delta -\lambda)\phi \|)
    \end{align}
    for any $\lambda$ and any $\phi$.   

    Finally, \cite[Proof Theorem 3.2]{Kontorovich2009} shows that, by testing against carefully selected test functions, if $G_T$ satisfies \eqref{GT bound} for any $\lambda$ and any $\phi$, then $G_T$ vanishes everywhere (the proof therein uses nothing else about $G_T)$.
      
  \end{proof}

   \subsection{Proof of Theorem \ref{thm:smooth m}}

  With the main identity at hand, we can proceed with the proof of Theorem \ref{thm:smooth m}.  By Parseval's identity \eqref{API}  
  \begin{align*}
   \wt{N}(T) &= \< F_T, \Psi \>_\Gamma\\
   &= \< \wh{F}_T, \wh{\Psi} \>_{\Spec(\Gamma)}\\
   &= \wh{F_T}(\lambda_0) \wh{\Psi}(\lambda_0) + \wh{F_T}(\lambda_1) \wh{\Psi}(\lambda_1) + \dots + \wh{F_T}(\lambda_{k}) \wh{\Psi}(\lambda_{k}) \\
   &\phantom{++++++++++++++}+ \int_{\cS^{temp}} \wh{F_T}(\lambda) \wh{\Psi}(\lambda) \rd\wh{\mu}(\lambda).
  \end{align*}
  Now for each point in the ``exceptional'' spectrum, $\wh{\Psi}(\lambda_i)$ is the projection onto the $i^{th}$ eigenspace, which is finite dimensional. Thus
  \begin{align*}
    \wh{\Psi}(\lambda_i) &= \<\Psi,\phi_i\>,
  \end{align*}
  and
  by the mean value theorem, we have that
  \begin{align*}
    \wh{\Psi}(\lambda_i) &= C_i+O(\vep).
  \end{align*}
  Moreover using \eqref{K bounds}, we have that  
  \begin{align*}
    \wh{F_T}(\lambda_i) &= T^{m-s_i} \<F_{1},\phi_i\>,\\
    &= c_i T^{m-s_i},
  \end{align*}
  for some constants $c_{i}$.

  As for the error term, we can use 
  \eqref{eq:Ldiag}
  %the fact that the Casimir operators act like scalars on irreducibles 
  and \eqref{K bounds} to achieve the following bound
  \begin{align*}
    \int_{\cS^{temp}} \wh{F_T}(\lambda) \wh{\Psi}(\lambda) \rd\wh{\mu}(\lambda) &=  \int_{\cS^{temp}} \wh{K_{T}(\Delta)F_T}(\lambda) \wh{\Psi }(\lambda) \rd\wh{\mu}(\lambda)\\
    %&=  
    %\int_{\cS^{temp}} K_{T}(\Delta) \wh{F_{1}}(\lambda) \wh{\Psi }(\lambda) \rd\wh{\mu}(\lambda)\\
    &=  \int_{\cS^{temp}} K_{T}(\lambda)\wh{F_{1}}(\lambda) \wh{\Psi }(\lambda) \rd\wh{\mu}(\lambda)\\
    &\ll T^{m/2} \int_{\cS^{temp}} \wh{F_{1}}(\lambda) \wh{\Psi }(\lambda) \rd\wh{\mu}(\lambda).
  \end{align*}
  From here we again apply Parseval's identity and Cauchy-Schwarz yielding
  \begin{align*}
    \int_{\cS^{temp}} \wh{F_1}(\lambda) \wh{\Psi}(\lambda) \rd\wh{\mu}(\lambda)
    &\ll T^{m/2}\|F_{1}\|\|\Psi\|.
  \end{align*}
  Since $\Gamma\bk G$ has finite volume and $F_{1}$ is bounded, the  $L^2$ norm of $F_{1}$ is bounded. 
  The proof follows on using \eqref{eq:PsiL2GL3}.

  \small 
  \section*{Acknowledgements}

  Kontorovich is partially supported by NSF grant DMS-2302641, BSF grant 2020119. The authors would like to thank Valentin Blomer, Alex Gorodnik, Zeev Rudnick, and the referees for comments and suggestions on an earlier draft. 

  \bibliographystyle{alpha}
  \bibliography{biblio}

\begin{thebibliography}{MgW95}

\bibitem[BKS10]{BourgainKontorovichSarnak2010}
J.~Bourgain, A.~Kontorovich, and P.~Sarnak.
\newblock Sector estimates for hyperbolic isometries.
\newblock {\em Geometric and Functional Analysis}, 20(5):1175--1200, Nov 2010.

\bibitem[Chr56]{Christopher1956}
J.~Christopher.
\newblock The asymptotic density of some {$k$}-dimensional sets.
\newblock {\em Amer. Math. Monthly}, 63:399--401, 1956.

\bibitem[DRS93]{DukeRudnickSarnak1993}
W.~Duke, Z.~Rudnick, and P.~Sarnak.
\newblock Density of integer points on affine homogeneous varieties.
\newblock {\em Duke Math. J.}, 71(1):143--179, 1993.

\bibitem[EM93]{EskinMcMullen1993}
A.~Eskin and C.~McMullen.
\newblock Mixing, counting, and equidistribution in {L}ie groups.
\newblock {\em Duke Math. J.}, 71(1):181--209, 1993.

\bibitem[GH85]{GoldfeldHoffstein1985}
D.~Goldfeld and J.~Hoffstein.
\newblock Eisenstein series of {${1\over 2}$}-integral weight and the mean
  value of real {D}irichlet {$L$}-series.
\newblock {\em Invent. Math.}, 80(2):185--208, 1985.

\bibitem[GN12]{GorodnikNevo2012}
A.~Gorodnik and A.~Nevo.
\newblock Counting lattice points.
\newblock {\em J. Reine Angew. Math.}, 663:127--176, 2012.

\bibitem[GNY17]{GorodnikNevoYehoshua2017}
A.~Gorodnik, A.~Nevo, and G.~Yehoshua.
\newblock Counting lattice points in norm balls on higher rank simple {L}ie
  groups.
\newblock {\em Math. Res. Lett.}, 24(5):1285--1306, 2017.

\bibitem[Gol06]{Goldfeld2006}
D.~Goldfeld.
\newblock {\em Automorphic forms and {$L$}-functions for the group {${\rm
  GL}(n,\bold R)$}}, volume~99 of {\em Cambridge Studies in Advanced
  Mathematics}.
\newblock Cambridge University Press, Cambridge, 2006.
\newblock With an appendix by K. Broughan.

\bibitem[KL24]{KontorovichLutsko2022}
A.~Kontorovich and C.~Lutsko.
\newblock Effective counting in sphere packings.
\newblock {\em J. of the Assoc. Math. Res.}, 2:15--52, 2024.

\bibitem[Kna01]{Knapp2001}
A.~Knapp.
\newblock {\em Representation theory of semisimple groups}.
\newblock Princeton Landmarks in Mathematics. Princeton University Press,
  Princeton, NJ, 2001.
\newblock An overview based on examples, Reprint of the 1986 original.

\bibitem[Kon09]{Kontorovich2009}
A.~Kontorovich.
\newblock The hyperbolic lattice point count in infinite volume with
  applications to sieves.
\newblock {\em Duke Math. J.}, 149(1):1--36, 2009.

\bibitem[Lan76]{Langlands1976}
Robert~P. Langlands.
\newblock {\em On the functional equations satisfied by {E}isenstein series}.
\newblock Lecture Notes in Mathematics, Vol. 544. Springer-Verlag, Berlin-New
  York, 1976.

\bibitem[LP82]{LaxPhillips1982}
P.~Lax and R.~Phillips.
\newblock The asymptotic distribution of lattice points in {E}uclidean and
  non-{E}uclidean spaces.
\newblock {\em J. Functional Analysis}, 46(3):280--350, 1982.

\bibitem[Lut22]{Lutsko2022}
C.~Lutsko.
\newblock An abstract spectral approach to horospherical equidistribution.
\newblock {\em arXiv:2211.01900}, 2022.

\bibitem[Mar04]{Margulis1970}
G.~A. Margulis.
\newblock {\em On some aspects of the theory of {A}nosov systems}.
\newblock Springer Monographs in Mathematics. Springer-Verlag, Berlin, 2004.
\newblock With a survey by Richard Sharp: Periodic orbits of hyperbolic flows,
  Translated from the Russian by Valentina Vladimirovna Szulikowska.

\bibitem[MgW95]{MoeglinWaldspurger1995}
C.~M\oe~glin and J.-L. Waldspurger.
\newblock {\em Spectral decomposition and {E}isenstein series}, volume 113 of
  {\em Cambridge Tracts in Mathematics}.
\newblock Cambridge University Press, Cambridge, 1995.
\newblock Une paraphrase de l'\'{E}criture [A paraphrase of Scripture].

\bibitem[Rud73]{Rudin1973}
W.~Rudin.
\newblock {\em Functional analysis}.
\newblock McGraw-Hill Series in Higher Mathematics. McGraw-Hill Book Co., New
  York-D\"{u}sseldorf-Johannesburg, 1973.

\bibitem[Ten95]{Tenenbaum1995}
G\'{e}rald Tenenbaum.
\newblock {\em Introduction to analytic and probabilistic number theory},
  volume~46 of {\em Cambridge Studies in Advanced Mathematics}.
\newblock Cambridge University Press, Cambridge, 1995.
\newblock Translated from the second French edition (1995) by C. B. Thomas.

\bibitem[Wu02]{Wu2002}
J.~Wu.
\newblock On the primitive circle problem.
\newblock {\em Monatsh. Math.}, 135(1):69--81, 2002.

\bibitem[You09]{Young2009}
M.~Young.
\newblock The first moment of quadratic {D}irichlet {$L$}-functions.
\newblock {\em Acta Arith.}, 138(1):73--99, 2009.

\end{thebibliography}

    \hrulefill

    \vspace{4mm}
     \noindent Department of Mathematics, Rutgers University, Hill Center - Busch Campus, 110 Frelinghuysen Road, Piscataway, NJ 08854-8019, USA. \emph{E-mail: \textbf{alex.kontorovich@rutgers.edu}}

    \vspace{4mm}
     \noindent Department of Mathematics,
     University of Houston, PGH Hall, 3551 Cullen Blvd, Houston, TX 77204, USA
     \emph{E-mail: \textbf{clutsko@uh.edu}}

\end{document}